\documentclass{article}
\usepackage[utf8]{inputenc}

\usepackage{amsmath}
\usepackage{graphicx}
\usepackage{amsfonts}
\usepackage{amssymb,enumerate}
\usepackage{amsthm}
\usepackage[all]{xy}
\usepackage{hyperref}
\theoremstyle{plain}
\newtheorem{lem}{Lemma}[section]

\theoremstyle{definition}

\newcommand{\bbz}{\mathbb{Z}}

\renewcommand{\geq}{\geqslant}
\renewcommand{\leq}{\leqslant}

\newtheorem*{mainthm*}{Main Theorem}

\begin{document}
\bibliographystyle{alpha}

\title 
{$\tau_I$-Elasticity for quotients of order four}
\author{Kailey B. Perry}
\date{}

\maketitle
\begin{abstract}
For a commutative domain $R$ with nonzero identity and $I$ an ideal of $R$, we say $a=\lambda b_1 \cdots b_k$ is a $\tau_I$-factorization of $a$ if $\lambda \in R$ is a unit and $b_i \equiv b_j$(mod $I$) for all $1\leq i \leq j \leq k$. 
These factorizations are nonunique, and two factorizations of the same element may have different lengths. In this paper, we determine the smallest quotient $R/I$ where $R$ is a unique factorization domain, $I\subset R$ an ideal, and $R$ contains an element with atomic $\tau_I$-factorizations of different lengths.
In fact, for $R=\mathbb{Z}[x]$ and $I = (2,x^2+x)$, we can find a sequence of elements $a_i$ that have an atomic $\tau_I$-factorization of length 2 and one of length $i$ for $i\in\mathbb{N}$.
\end{abstract}

\section{Introduction}

For a commutative domain $R$ with nonzero identity and $I$ an ideal of $R$, we say $a=\lambda b_1 \cdots b_k$ is a \emph{$\tau_I$-factorization} of $a$ if $\lambda \in R$ is a unit and $b_i \equiv b_j$(mod $I$) for all $1\leq i \leq j \leq k$. An element $a$ is a \emph{$\tau_I$-atom} if the only $\tau_I$-factorizations of $a$ are of length one. An element $a$ is \emph{$\tau_I$-atomic} if it has a $\tau_I$-factorization into $\tau_I$-atoms. A domain $R$ with an ideal $I$ is \emph{$\tau_I$-atomic} if every nonzero nonunit in $R$ is $\tau_I$-atomic.

For example, for $R = \mathbb{Z}$ and $I=(3)$,
$20=2\cdot 2 \cdot 5$ is a $\tau_{(3)}$-atomic factorization as well as the prime factorization, because all prime factors of 20 are congruent modulo 3. Notice that
$28=4\cdot 7$ is not a $\tau_{(3)}$-atomic factorization, even though $4 \equiv 7$(mod 3), because 4 is not a $\tau_{(3)}$-atom. Instead, $28 = -1\cdot 2\cdot 2\cdot (-7)$ is a $\tau_{(3)}$-atomic factorization, using the unit $\lambda = -1$.

The topic of $\tau$-factorization was introduced by Anderson and Frazier \cite{anderson_general_2011}, and has since been studied by Florescu \cite{florescu_reduced_2013}, Hamon \cite{Hamon2007SomeTI}, Juett \cite{juett_topics_2013}, Mahlum and Mooney \cite{mahlum_generalized_2016}, Mooney \cite{mooney_tau_2016}, Anderson and Ortiz-Albino \cite{anderson_three_2012}, and Anderson and Reinkoester \cite{anderson_generalized_nodate}, as well as others.

Originally defined by Anderson and Frazier \cite{anderson_general_2011} only for elements of $\tau_I$-atomic domains, the definition of $\tau_I$-elasticity was expanded by Hasenauer, Kubik and Perry \cite{Taustuff} to also apply to elements with a $\tau_I$-factorization into atoms in domains that are not $\tau_I$-atomic. We will be using the expanded definition. 

We will say that a UFD $R$ is a \emph{$\tau_I$-half factorial domain} if all $\tau_I$-factorizations of a given element into atoms have the same length. This is also using the expanded notion of elasticity, only considering atomic elements rather than requiring the entire domain be atomic.
For $m \in R$ a $\tau_I$-atomic element, we define the \emph{$\tau_I$-elasticity} of $m$ to be 
$$\rho_I(m)=\sup\{\frac{k}{l}: m=a_1\cdots a_k=b_1\cdots b_l\text{  both $\tau_I$-factorizations into atoms}\}$$

\noindent and the \emph{$\tau_I$-elasticity} of $R$ is defined 

$$
\rho_I(R)=\sup\{\rho_I(m): m\in R, m \text{ is $\tau_I$-atomic}\}.
$$

In this paper, we ask, given any unique factorization domain $R$ and an ideal $I\subset R$, what is the smallest quotient $R/I$ such that $R$ fails to be a $\tau_I$-half factorial domain, i.e. that $\rho_I(R)>1$.
We show a UFD $R$ with an ideal $I$ such that $|R/I|=4$ and $R$ has infinite $\tau_I$-elasticity.
Hasenauer and Kubik found that 
four is also the smallest order for a quotient $R/I$ where $R$ fails to be $\tau_I$-atomic \cite{hasenauer_tau-atomicity_2021}. Interestingly, it is a different quotient that turns out to have infinite elasticity. This is in part due to the fact that we consider $R$ any UFD, where Hasenauer and Kubik only took $R$ to be a principal ideal domain. When $R$ is only a UFD, none of the isomorphism classes of order 4 need be $\tau_I$-atomic.

Our result is as follows.

\begin{mainthm*}
Let R be a UFD and I an ideal of R with $|R/I| = 4$. The $\tau_I$-elasticity of R is infinite when $R/I \cong \bbz[x]/(2, x^2+x)$ and there is at least one prime in $R$ in each of the non-unit equivalence classes of the quotient, and the $\tau_I$-elasticity of $R$ is 1 otherwise.
\end{mainthm*}

Because there is only one isomorphism class of commutative rings with unit of each order 1, 2 and 3, and it was shown by Hasenauer, Kubik and Perry \cite{Taustuff} that $\rho_{(n)}(\bbz) = 1$ for $n= 1,2,3$, this theorem shows that four is the smallest quotient $R/I$ for which $R$ is not a $\tau_I$-half factorial domain, and is the smallest quotient for which $\rho_I(R) = \infty$.

%%%%%%%%%%%%%%%%%%%%%%%%%

\section{Proof of Main Theorem}

There are four commutative rings with identity of order four (see the problem and solution by Singmaster and Bloom for more \cite{10.2307/2312421}).
The four isomorphism classes are $\mathbb{Z}/4\mathbb{Z}$, $\mathbb{Z}[x]/(2,x^2+1)$, $\mathbb{Z}[x]/(2,x^2+x+1)$, and $\mathbb{Z}[x]/(2,x^2+x)$.
We will prove the Main Theorem using the following four lemmas, considering each isomorphism class individually.

Our first lemma is a generalization of a result from Hasenauer, Kubik and Perry\cite{Taustuff}. In that paper, it is shown that $\rho_{(4)}(\bbz) = 1$. Now, we show that $\rho_I(R) = 1$ for any UFD $R$ and ideal $I$ such that $R/I \cong \bbz/4\bbz$.

\begin{lem}
Let R be a UFD and I an ideal of R. If $R/I \cong \bbz/4\bbz$, then $\rho_I(a)=1$ for all $\tau_I$-atomic $a\in R$, and $R$ is a $\tau_I$-HFD.
\end{lem}

\begin{proof}
Let $\tilde{f}: \bbz/4\bbz \to R/I$ be an isomorphism and denote by $f(i)$ some element of $R$ in the $\tilde{f}(i)$ isomorphism class for each $i\in \bbz/4\bbz$.

Let $a\in R$ and $a=p_1 \cdots p_k q_1 \cdots q_l r_1 \cdots r_m s_1 \cdots s_n$ be the unique factorization of $a$ into primes, such that $p_i \equiv f(1)$ (mod $I$), $q_i \equiv f(2)$ (mod $I$), $r_i \equiv f(3)$ (mod $I$), and $s_i \equiv f(0)$ (mod $I$).  
Because the equivalence class of 2, $\bar{2} \in \bbz/4\bbz$ is a zero divisor, there are no units congruent to $f(2)$ modulo $I$, and there cannot be any units congruent to $f(0) \in I$. So any units of $R$ can only be congruent to $f(1)$ modulo $I$ or $f(3)$ modulo $I$. Notice then, that for any unit $\lambda \in R$, and any $q_i \equiv f(2)$ (mod $I$), $\lambda q_i \equiv f(2)$ (mod $I$), and for any unit $\lambda$ and any $p_i \equiv f(1)$ (mod $I$) or $r_i \equiv f(3)$ (mod $I$), $\lambda p_i \not\equiv f(2)$ (mod $I$) and $\lambda r_i \not\equiv f(2)$ (mod $I$). Further, because $1\equiv f(1)$ (mod $I$) and $-1\equiv f(3)$ (mod $I$), there is at least one unit in each of those equivalence classes.

Because there is no element in $\bbz/4\bbz$ that, when raised to a power greater than one, is equal to $\Bar{2}$, any product of primes in $R$ that is congruent to $f(2)$ modulo $I$ must be an atom. So for $a\in R$ congruent to $f(2)$ modulo $I$, $a$ is an atom, so the length of any $\tau_I$-factorization of $a$ is 1.

 Because any product $p_1 \cdots p_k q_1\cdots q_l r_1 \cdots r_m$ for $l \geq 2$ could have a $\tau_I$-factorization of $(q_1 p_1 \cdots p_k r_1 \cdots r_s)q_2\cdots q_l$, such a product cannot be an atom. So, an atom which contains no primes in $I$ cannot be congruent to 0 modulo $I$. 
 A product $p_1 \cdots p_k q_1\cdots q_l r_1 \cdots r_m s_1\cdots s_n$ for $n\geq 2$ can also be $\tau_I$-factored as $(p_1 \cdots p_k q_1\cdots q_l r_1 \cdots r_m s_1) s_2\cdots s_n$. 
 If instead $n=1$ and $l\geq 2$, it can be $\tau_I$-factored as $(p_1 \cdots p_k  r_1 \cdots r_m s_1)( q_1\cdots q_l)$ and is not an atom. 
 If $n=1$ and $l=1$, the product $p_1 \cdots p_k q_1 r_1 \cdots r_m s_1$ is an atom, because if it were to be factored, only the factor containing $s_1$ would be congruent to $0$ (mod $I$).

 Let $a\in R$ be congruent to 0 modulo $I$, with a prime factorization of $p_1 \cdots p_k q_1\cdots q_l r_1 \cdots r_m s_1\cdots s_n$. Since $0+I$ and $f(2)+I$ are the only zero divisors in $R/I$, a $\tau_I$-factorization must have atoms congruent to $f(2)$ modulo $I$ or $0$ modulo $I$. 
 If $n \geq 1$, then all atoms in a $\tau_I$-factorization of $a$ must be congruent to $0$ modulo $I$, and since an atom congruent to $0$ (mod $I$) must contain exactly one prime congruent to $0$ (mod $I$), the length of any $\tau_I$-factorization of $a$ is $n$.
 If $n=0$, because all atoms in a $\tau_I$-factorization of $a$ must be congruent to 2 modulo $I$, each atom must contain exactly one prime congruent to $f(2)$ modulo $I$, and the length of any $\tau_I$-factorization must be $l$.
 
 If a product of primes in $R$ is congruent to $f(1)$ modulo $I$ or $f(3)$ modulo $I$, then none of its factors are congruent to $f(2)$  or $0$ modulo $I$. Since $-1$ is a unit in $R$, for $r\equiv f(3)$(mod I), $-r$ is congruent to $f(1)$ modulo $I$. So, a product $p_1 \cdots p_k r_1 \cdots r_m$ can be factored to $(-1)^m p_1 \cdots p_k (-r_1)\cdots(-r_m)$, and cannot be an atom for $k+m \geq 2$. So for $a\in R$ with prime factorization $a=p_1 \cdots p_k r_1 \cdots r_m$, the $\tau_I$-factorization of $a$ can only have length $k+m$.
 
 Because the length of a $\tau_I$-factorization is unique for any $\tau_I$-atomic $a\in R$, the elasticity $\rho_I(a) =1$ for all $\tau_I$-atomic elements $a\in R$, and so $\rho_I(R)=1$.
\end{proof}

Notice that this proof used the fact that $\tau_I$-elasticity is only defined over $\tau_I$-atomic elements of $R$. We do not consider elements like $a = q_1q_2 s_1$, because these do not have a $\tau_I$-factorization into atoms.

The proof of the following lemma is very similar to the preceding one, as $x+1 \in \bbz[x]/(2, x^2+1)$ is the only nonzero zero-divisor, squares to 0, and for every other nonzero element $a \in \bbz[x]$, $(x+1)a \equiv x+1$ (mod $(2,x^2+1)$), and so behaves like $2\in \bbz/4\bbz$.

\begin{lem}
Let R be a UFD and I an ideal of R. If $R/I \cong \bbz[x]/(2, x^2+1)$, then  $ \rho_I(a)=1$ for all $\tau_I$-atomic $a\in R$ and so R is a $\tau_I$-HFD.
\end{lem}

\begin{proof}
A Cayley table showing multiplication in $\bbz[x]/(2, x^2+1)$ is provided.

\begin{center}
    \begin{tabular}{c|c c c}
         & $1$ & $x$ & $x+1$  \\
        \hline
        $1$       & $1$         & $x$       & $x+1$ \\
        $x$       & $x$         & $1$       & $x+1$  \\
        $x+1$     & $x+1$       & $x+1$     & $0$     \\
    \end{tabular}   
\end{center}

Let $\tilde{f}: \bbz[x]/(2, x^2+1) \to R/I$ be an isomorphism, and let $f(a)\in R$ be some element that the quotient map sends to $\tilde{f}(a)$ for each $a\in \bbz[x]/(2, x^2+1)$.
Let $a\in R$ and $a=p_1 \cdots p_k q_1 \cdots q_l r_1 \cdots r_m s_1 \cdots s_n$ be the unique factorization of $a$ into primes, such that $p_i \equiv f(1)$ (mod $I$), $q_i \equiv f(x)$(mod $I$), $r_i \equiv f(x+1)$(mod $I$), and $s_i \equiv f(0)$ (mod $I$). 
Since $f(x+1)+I \in R/I$ is a zero divisor, there cannot be any unit congruent to $f(x+1)$ modulo $I$ or congruent to $f(0)$ modulo $I$. Notice that for any unit in $f(1)+I$ or in $f(x)+I$, multiplication by a unit does not change the class of any prime in $f(x+1)+I$. 

Notice that $\tilde{f}(x+1) \in R/I$ here behaves like $\bar{2}\in \bbz/4\bbz$ in the previous lemma in that it squares to 0, is the only nonzero zero divisor, and absorbs other products. So, for the same reasons as in the previous lemma, if $a$ is congruent to $f(x+1)$ modulo $I$ or $f(0)$ modulo $I$, each atom in any $\tau_I$-factorization must contain exactly one prime congruent to $f(x+1)$ modulo I if it contains any and if it contains no primes congruent to $f(0)$ modulo $I$. The length of any $\tau_I$-factorization of $a$ then is $n$ if $n$ is nonzero and is $m$ is $n=0$ and $m\neq 0$.

With $n=m=0$, a product $p_1\cdots p_k q_1 \cdots q_l$ with $l\geq 2$ can have a $\tau_I$-factorization $(q_1 p_1\cdots p_k)q_2 \cdots q_l$, so such a product cannot be an atom. When there are no prime factors congruent to $f(x+1)$ modulo $I$, there can be at most one prime factor congruent to $f(x)$ modulo $I$ in an atom.

If there exists a unit $\lambda \in R$ congruent to $f(x)$ modulo $I$, a product

\noindent$p_1\cdots p_k q_1 \cdots q_l$ would factor to $\lambda^{-l} p_1 \cdots p_k (\lambda q_1)\cdots (\lambda q_l)$. For 

\noindent$a = p_1\cdots p_k q_1 \cdots q_l$, an atom couldn't contain more than one prime factor, and the length of any $\tau_I$-factorization would be $k+l$.

If there are no units in $R$ congruent to $f(x)$ modulo $I$, the only units are congruent to $f(1)$ modulo $I$ and so, multiplication by a unit cannot change the equivalence class of a prime. So, if $a=p_1\cdots p_k q_1 \cdots q_l$ for $l\geq 1$, each atom must contain exactly one prime congruent to $f(x)$ modulo $I$ and the length of any $\tau_I$-factorization must be $l$.

 Because the length of a $\tau_I$-factorization is unique for any $a\in R$, the elasticity $\rho_I(a) =1$ for all elements $a\in R$, and so $\rho_I(R)=1$.
\end{proof}

%%%%%%%%%%%%% 

\begin{lem}
Let R be a UFD and I an ideal of R. If $R/I \cong \mathbb{F}_4$ the field of order four, then  $\rho_I(w) = 1$ for all $\tau_I$-atomic $w\in R$ and R is a $\tau_I$-HFD.
\end{lem}

\begin{proof}
A Cayley table showing multiplication in the field of order 4 is provided, as  $\mathbb{F}_4\cong \bbz[x]/(2, x^2+x+1)$.

\begin{center}
    \begin{tabular}{c|c c c}
         & $1$ & $x$ & $x+1$  \\
        \hline
        $1$       & $1$         & $x$       & $x+1$ \\
        $x$       & $x$         & $x+1$       & $1$  \\
        $x+1$     & $x+1$       & $1$     & $x$     \\
    \end{tabular}   
\end{center}

Let $\tilde{f}: \bbz[x]/(2, x^2+x+1)\to R/I$ be an isomorphism, and let $f(a)\in R$ be some element that the quotient map sends to $\tilde{f}(a)$ for each $a\in \bbz[x]/(2, x^2+x+1)$.
As shown by Hasenauer and Kubik \cite{hasenauer_tau-atomicity_2021}, $R$ is only $\tau_I$-atomic if there is a unit in every class or if R does not contain a prime in both $f(x)+I$ and $f(x+1)+I$. We begin by assuming $R$ is $\tau_I$-atomic.

If there is a unit in every class, then the prime factorization of any $w \in R$ will be a $\tau_I$-factorization with units where necessary, and $\rho_I(R) = 1$.

Otherwise, let $w$ be an element of $R$ with $w = p_1\cdots p_k q_1 \cdots q_l r_1 \cdots r_m s_1 \cdots s_n$ a prime factorization of $w$ with $p_i \equiv f(0)$ (mod $I$), $q_i \equiv f(1) $ (mod $I$), $r_i \equiv f(x)$ (mod $I$) and $s_i \equiv f(x+1)$ (mod $I$). 
In this case, we will assume $n=0$ because there are only primes in $x+I$. The case where there are only primes in $x+1+I$ is similar.

If $w\equiv f(0)$ (mod $I$), then $k\neq 0$. Any atom in a $\tau_I$-factorization of $w$ must contain at least one prime factor congruent to $f(0)$ (mod $I$), because there are no nonzero zero divisors in $R/I$. Any product $p_1\cdots p_k q_1 \cdots q_l r_1\cdots r_m$ with $k >1$ could be factored to $p_1\cdots p_{k-1}(p_kq_1 \cdots q_s r_1\cdots r_t)$, so an atom could contain no more than one factor of $p_i$. So, the length of any $\tau_I$-factorization of $w$ must be exactly $k$.

If $k=0$, $w=q_1 \cdots q_l r_1\cdots r_m$. Similarly to the previous lemma, because the only primes are congruent to $f(1)$ modulo $I$ or $f(x)$ modulo $I$, there must be exactly one prime congruent to $f(x)$ modulo $I$ in each atom. The length of any $\tau_I$-factorization of $w$, then, must be $m$ if nonzero, and $l$ if $m=0$.

In either case, if $R$ is $\tau_I$-atomic, then for any $w \in R$, $\rho_I(w) =1$.

By using the definition of elasticity used by Hasenauer, Kubik and Perry \cite{Taustuff}, we can also consider the elasticity of an element of $R$ that is atomic, even when $R$ is not $\tau_I$-atomic itself. As shown by Hasenauer and Kubik \cite{hasenauer_tau-atomicity_2021}, $R$ is not $\tau_I$-atomic when $R/I \cong \mathbb{F}_4$, $R/I$ has a prime in every class, and all units of $R$ are congruent to $f(1)$ modulo $I$. To use this definition of elasticity, we will have to know exactly which elements of $R$ are $\tau_I$-atomic. Elements congruent to $f(0)$ modulo $I$ are atomic and, as shown above, can only be factored with one factor of $p_i$ in each atom. If $m=0$ or $n=0$, then $w$ factors into atoms as shown above.

We will show that $w=q_1\cdots q_l r_1\cdots r_m s_1 \cdots s_n$ can only be $\tau_I$-atomic when $m=n$, by contradiction.
Let $m>n$ (the case when $n>m$ is similar). If $w$ were to have a $\tau_I$-factorization, then there would have to be at least one atom which contained more primes congruent to $f(x)$ modulo $I$ than those congruent to $f(x+1)$ modulo $I$. 
We will represent this atom $r_1 \cdots r_\alpha s_1 \cdots s_\beta$. If $\alpha$ were even, this product would factor as $(r_1 r_2)\cdots(r_{\alpha -1} r_\alpha)s_1 \cdots s_\beta$, and so wouldn't be an atom. 
If $\beta$ were even, similarly, the product would factor as $r_1 \cdots r_\alpha (s_1 s_2)\cdots (s_{\beta -1} s_\beta)$ and wouldn't be an atom. 
Now, if $\alpha$ and $\beta$ were both odd and $\alpha \geq 5$, then the product could factor as 

\noindent$(r_1 r_2 r_3 r_4 r_5)(r_6 r_7)\cdots(r_{\alpha -1} r_\alpha)s_1 \cdots s_\beta$, and could not be an atom. Finally, if $\alpha$ and $\beta$ were both odd and $5>\alpha >\beta$, then $\alpha=3$ and $\beta =1$. But $r_1 r_2 r_3 s_1$ can factor as $r_1(r_2 r_3 s_1)$, and so is not an atom.

So, for $w=q_1\cdots q_l r_1 \cdots r_m s_1 \cdots s_n$ to be $\tau_I$-atomic, it must be that $m=n$, and the atoms of the $\tau_I$-factorization can only contain one prime congruent to each of $f(x)$ modulo $I$ and $f(x+1)$ modulo $I$. Any $\tau_I$-factorization, then, has length $m=n$.

So, in any of these cases where $R/I \cong \mathbb{F}_4$, the elasticity of any atomic $w \in R$ is $\rho_I(w) =1$, so $\rho_I(R)=1$.

\end{proof}

So we have seen that for $R$ a UFD with $I$ an ideal and $|R/I|=4$, for three of the four isomorphism classes of the quotient, we have $\rho_I(R)=1$. The following lemma will finish the proof of the theorem by showing that for the fourth isomorphism class, in certain cases, we have $\rho_I(R) = \infty$.

\begin{lem}
Let R be a UFD and I an ideal of R. If $R/I \cong \bbz[x]/(2, x^2+x)$ and there is at least one prime in each of the $x+I$ and $x+1+I$ classes, then  $\rho_I(R)= \infty$.
Otherwise, $\rho_I(R) =1$.
\end{lem}

\begin{proof}
A Cayley table showing multiplication in $\bbz[x]/(2, x^2+x)$ is provided.

\begin{center}
    \begin{tabular}{c|c c c}
         & $1$ & $x$ & $x+1$  \\
        \hline
        $1$       & $1$         & $x$       & $x+1$ \\
        $x$       & $x$         & $x$       & $0$  \\
        $x+1$     & $x+1$       & $0$     & $x+1$     \\
    \end{tabular}   
\end{center}

Let $f: \bbz[x]/(2, x^2+x)\to R/I$ be an isomorphism, and let $f(a)\in R$ be some element that the quotient map sends to $\tilde{f}(a)$ for each $a\in \bbz[x]/(2, x^2+x)$.

Let $a\in R$ and $a=p_1 \cdots p_k q_1 \cdots q_l r_1 \cdots r_m s_1\cdots s_n$ be the unique factorization of $a$ into primes, such that $p_i \equiv f(1)$ (mod $I$), $q_i \equiv f(x)$ (mod $I$), $r_i \equiv f(x+1)$ (mod $I$), and $s_i \equiv f(0)$ (mod $I$). Since $f(x+1)+I$ and $f(x)+I$ are zero divisors in $R/I$, all units of R must be congruent to $f(1)$ modulo $I$.

In the case that $a$ is not congruent to $f(0)$ modulo $I$, then $n=0$ and either $m=0$ or $l=0$. We will assume $n=m=0$ and $l>0$, as the case where all three are zero is obvious. The case where $n=l=0$ and $m>0$ is similar.
Because $l>0$ and the only prime factors of $a$ are congruent to $f(x)$ modulo $I$ or $f(1)$ modulo $I$, every atom must contain exactly one term congruent to $f(x)$ modulo $I$. So, the length of any $\tau_I$-factorization is exactly $l$.

Now suppose $a \in I$. Then $n>0$ or both $l>0$ and $m>0$. Because the only units are congruent to $f(1)$ modulo $I$ and no elements of $R/I$ are nilpotent, every atom must be congruent to $f(0)$ modulo $I$.
So every atom must contain either one prime congruent to $f(0)$ modulo $I$ or at least one prime congruent to $f(x)$ modulo $I$ and at least one congruent to $f(x+1)$ modulo $I$. 
Any product $q_1 \cdots q_\alpha r_1 \cdots r_\beta$ with $\alpha >1$ and $\beta >1$ could factor to $(q_1 r_1)(q_2\cdots q_\alpha r_1\cdots r_\beta)$, and so couldn't be an atom. So, each atom must have exactly one prime congruent to $f(0)$ modulo $I$, exactly one prime congruent to $f(x)$ modulo $I$ or exactly one prime congruent to $f(x+1)$ modulo $I$. 
Notice that any product $q_1 r_1 \cdots r_\beta$ or $r_1 q_1\cdots q_\alpha$ must be an atom, as there is no way to create two or more factors that are both congruent to 0 modulo $I$, for any value of $\alpha$ and $\beta$ that is greater than or equal to one.

So, for any $a= p_1\cdots p_k q_1\cdots q_l r_1\cdots r_m \in I$ with no prime factors congruent to $f(0)$ (mod $I$), $a$ has a shortest $\tau_I$-factorization of length 2, with one atom containing one prime congruent to $f(x)$ modulo $I$ and $m-1$ primes congruent to $f(x+1)$ modulo $I$, and the other atom containing $l-1$ primes congruent to $f(x)$ modulo $I$ and one prime congruent to $f(x+1)$ modulo $I$. The inclusion of primes congruent to $f(1)$ (mod $I$) does not affect whether these are atoms. Also, $a$ has a longest $\tau_I$-factorization with length $\min(l,m)$, with each atom containing exactly one prime congruent to either $f(x)$ modulo $I$ or $f(x+1)$ modulo $I$, whichever is appropriate.

So, for such elements $a \in I$, $\rho_I(a)=\frac{\min(l,m)}{2}$. Because the length of the prime factorizations for such elements $a$ is unbounded, $\min(l,m)$ is unbounded as well, and $\rho_I(R)=\infty$.

\end{proof}

This concludes the proof of the theorem. Notice that $\bbz[x]$ with the ideal $(2, x^2+x)$ itself is such a ring $R$, as both $x,x+1$ are prime in $\bbz[x]$. Here, such elements $a\in \bbz[x]$ might be $x^l(x+1)^m$. So we have a sequence $\{a_i=x^i(x+1)^i\}$ for which the values $\rho_I(a_i)$ are unbounded.

Similar questions could be explored for quotients of larger order. It was found by Hasenauer, Kubik and Perry \cite{Taustuff} that $\rho_{(n)}(\mathbb{Z})$ is only finite and not equal to 1 in the cases where $n=12$ and $n=18$, in which cases  $\rho_{(n)}(\mathbb{Z}) = 2$. An interesting project could be finding the smallest quotient with finite elasticity not equal to 1. Another could be finding a ring and ideal such that $\rho_I(R)$ is not an integer, or even irrational.

\bibliography{bib}

\end{document}